\documentclass{birkjour}
\usepackage{mathtools,amsmath,amssymb}
\usepackage{hyperref}
\usepackage[nameinlink,capitalize]{cleveref}
\usepackage{multicol}
\usepackage{enumitem}
\Crefname{ass}{Ass.}{Ass.}
\Crefname{defn}{Def.}{Defs.}
\Crefname{thm}{Thm.}{Thms.}
\Crefname{lem}{Lem.}{Lems.}

\usepackage{pgfplots}
\pgfplotsset{compat=1.18}

\usepackage{todonotes}
\usepackage{comment}

\usepackage[]{xcolor}

\definecolor{col1}{rgb}{     0, 0.4470, 0.7410}
\definecolor{col2}{rgb}{0.8500, 0.3250, 0.0980}
\definecolor{col3}{rgb}{0.9290, 0.6940, 0.1250}
\definecolor{col4}{rgb}{0.4940, 0.1840, 0.5560}
\newcommand{\plotfont}{\fontsize{8}{10}\selectfont} %<-- New fontsize
 \newtheorem{thm}{Theorem}[section]
 
 \newtheorem{lem}[thm]{Lemma}
 
 \theoremstyle{definition}
 \newtheorem{defn}[thm]{Definition}
 \newtheorem{ass}[thm]{Assumption}
 \theoremstyle{remark}
 \newtheorem{rem}[thm]{Remark}
 
 \numberwithin{equation}{section}
\usepackage{graphicx} % Required for inserting images

\newcommand{\R}{\mathbb{R}}

\renewcommand{\epsilon}{\varepsilon}
\newcommand{\N}{\mathbb N}

\newcommand{\eps}{\epsilon}

\newcommand{\sL}{{\mathsf{L}\!}}
\newcommand{\sBV}{{\mathsf{BV}}}
\newcommand{\sW}{\mathsf W}
\newcommand{\sC}{\mathsf C}
\newcommand{\sTV}{\mathsf{TV}}

  \newcommand{\sgn}{\ensuremath{\textnormal{sgn}}}

\newcommand{\weakstar}{\overset{*}{\rightharpoonup}}

\newcommand{\OT}{{\Omega_{T}}}
\newcommand{\loc}{\;\textnormal{loc}}

\newcommand{\dd}{\ensuremath{\,\mathrm{d}}}

\date{\today}
\allowdisplaybreaks[4]
\begin{document}

\title[Nonlocal conservation laws and sign-unrestricted datum]
 {A note on nonlocal approximations of sign-unrestricted solutions of conservation laws}

%----------Author 1
\author[A. Keimer]{Alexander Keimer}

\address{%
Department of Mathematics\\
Friedrich-Alexander-University Erlangen-Nuremberg (FAU)\\
Cauerstraße 11\\
91058 Erlangen}

\email{alexander.keimer@fau.de}

%\thanks{THANK YOU for traveling with deutsche Bahn}
%----------Author 2
\author[L. Pflug]{Lukas Pflug}
\address{
Department of Mathematics\ \& Competence Center Scientific Computing\\
Friedrich-Alexander-University Erlangen-Nuremberg (FAU)\\
Cauerstraße 11\ \& Martensstraße 5a\\
91058 
Erlangen}
\email{lukas.pflug@fau.de}


\date{\today}
\dedicatory{}
%%% ----------------------------------------------------------------------

\begin{abstract}
We study the singular limit problem for nonlocal conservation laws in which the sign of the initial datum is unrestricted and the velocity of the conservation law depends on a nonlocal approximation of the absolute value of the density. We demonstrate that the nonlocal solutions converge to the local entropy solution when the nonlocal kernel tends to a Dirac distribution, and thus obtain an approximation result for local unsigned conservation laws, generalizing the current results on the so-called sign-restricted singular limit problem. The considered model class covers special cases like a generalized Burgers' equation and scalar versions of the Keyfitz--Kranzer system.
\end{abstract}

%%% ----------------------------------------------------------------------
\maketitle
%%% ----------------------------------------------------------------------
%\tableofcontents
\section{Introduction}
Nonlocal conservation laws have received increased attention in the last de\-cade as they can model a variety of important physical and human-driven dynamics \cite{amadori,degond,amorim, teixeira,piccoli2018sparse,colombo2012,bayen2022modeling,pflug2020emom,rossi2020well,friedrich2018godunov,kloeden,blandin2016well,chiarello,chiarello2019stability}, and also possess interesting mathematical properties that deviate from the theory of local conservation laws \cite{pflug,crippa2013existence,wang}.
Usually, the word ``nonlocal'' refers to the fact that the velocity of the conservation law depends not only on the solution at a given point in space and time but also on its integral in a neighborhood of this point.

In this work, we focus in particular on nonlocal conservation laws in which the nonlocality enters the flux nonlinearly, and the ``locality'' linearly; that is, when the flux can be decomposed into
\[
f(q,\gamma\ast q)=V(\gamma\ast q)\cdot q
\]
where \(\gamma\) is an integral kernel and \(\ast\) denotes the spatial convolution.

A question that has attracted great interest (first raised in \cite{teixeira}) is whether one can recover the corresponding local entropy solution when the nonlocal kernel converges to a Dirac distribution. Several publications have addressed this question 
\cite{zumbrun1999,pflug4,Crippa2021,colombo2023nonlocal,bressan2019traffic,bressan2021entropy,coclite2022general,friedrich2022conservation,coclite2023oleinik,chiarello2023singular,keimer2023singular,colombo2023overview,chiarello2023singular,keimer2023singulardiscontinuous}. For an overview, we refer the reader to \cite{colombo2023overview}.

So far, however, all results require the initial datum---and thus the solution---to be sign restricted. In this work, we will consider sign-unrestricted initial data for cases in which the previously mentioned velocity function \(V\) depends only on the absolute value of the density, i.e.,
\[V(x)=V(|x|)\quad \forall x\in I\]
where \(I\) represents all the solution's values. Then, we replace the nonlocal operator in the nonlocal approximation by its normed equivalent, i.e., 
\[
\gamma_{\eta}\ast q \text{ by } W[|q_{\eta}|,\gamma_{\eta}]\coloneqq\gamma_{\eta}\ast |q|.
\]
On \(\OT\), the local and nonlocal conservation laws read as
\[
\partial_{t}q+\partial_{x}\big(V(|q|)q\big)=0\qquad \text{and}\qquad \partial_{t}q_{\eta}+\partial_{x}\big(V(W[|q_{\eta}|,\gamma_{\eta}])q_{\eta}\big)=0.
\]

Under some additional constraints, that the involved kernels are one-sided and that the velocity function is monotone (these are classical assumptions in the considered model class and are not restrictive), we show that the local entropy solution of the corresponding local equation is recovered in the limit \(\gamma\rightarrow \delta\).
Before we turn to the analysis, we highlight some important special cases of this model class.
\paragraph{Keyfitz--Kranzer systems:}
The class includes the scalar case of the Keyfitz--Kranzer system \cite{risebro2013note}, i.e., 
\begin{align}
    \partial_t q + \partial_{x}(q \phi(\|q\|) ) &= 0
\end{align}
with $q : [0,T]\times \R \rightarrow \R^n$ for $n\in \mathbb N$ and $\phi \in C^1(\R_{\geq 0})$, which is of interest in a number of applications ranging from physics to engineering \cite{tveito1991existence}. This contribution is thus a first step toward approximating this type of system nonlocally while---in contrast to viscosity approximations---preserving the hyperbolic nature of the problem.

\paragraph{Generalized Burgers' equation:}
The approximation result is also applicable to the generalized Burgers' equation \cite{tersenov2010generalized}, i.e., 
\[
\partial_{t}q+ \partial_{x}\big(q^{2m+1}\big)=0
\]
with \(m\in\N_{\geq1}\) since in this case, the velocity is \(V\equiv (\cdot)^{2m}\equiv |\cdot|^{2m}\).
\section{Basic assumptions and problem setup}
We study the nonlocal version of the following scalar conservation law as a Cauchy problem on \(\R\):
\begin{equation}
\begin{aligned}
    \partial_{t}q+ \partial_{x}\big(V(|q|)q\big)&=0,&& (t,x)\in \OT\\
    q(0,\cdot)&\equiv q_{0}, && x\in\R
\end{aligned}
\label{eq:local_dynamics}
\end{equation}
for \(q:\OT\coloneqq(0,T)\times\R\rightarrow\R\) and \(V:\R\rightarrow\R\). For \(q_{\eta}:(0,T)\times\R\rightarrow\R\), its nonlocal approximation reads as
\begin{equation}
\begin{aligned}
    \partial_{t}q_{\eta}+ \partial_{x}\big(V(W[|q_{\eta}|,\gamma_{\eta}])q_{\eta}\big)&=0,&& (t,x)\in \OT\\
    W[|q_{\eta}|,\gamma_{\eta}](t,x)&=\int^{x}_{-\infty}\!\!\!\!\!\gamma_{\eta}(y-x)|q_{\eta}(t,y)|\dd y, &&(t,x)\in\OT\\
    q_{\eta}(0,\cdot)&\equiv q_{0}, && x\in\R
\end{aligned}
\label{eq:nonlocal_dynamics}
\end{equation}
for  a kernel \(\gamma_{\eta}:\R_{\leq 0}\rightarrow\R_{\geq0}\) yet to be specified.
We will prove that
\[
q_{\eta}\weakstar q^{*} \text{ in } \sL^{\infty}((0,T);\sL^{\infty}(\R)),\  W[|q_{\eta}|,\gamma_{\eta}]\rightarrow |q^{*}|\text{ in } \sC\big([0,T];\sL^{1}_{\loc}(\R)\big)
\]
where \(q^{*}\) is the (entropy) weak solution of the local conservation law in \cref{eq:local_dynamics}.

The fundamental difference from the nonlocal conservation laws studied so far is that we take the absolute value in the nonlocal term, so that we approximate the solution in the \(\sL^{1}\)-norm and can thus also only expect convergence in the nonlinearity to the absolute value. 
As stated before, the initial datum in the considered setup does not need to be sign restricted. 
All assumptions that will be used in this contribution are collected in the following.
\begin{ass}\label{ass:input_datum}
We make these assumptions for the involved datum:
%\begin{multicols}{2}
\begin{itemize}%[leftmargin=*]
    \item \(V\in \sW^{1,\infty}_{\loc}(\R_{>0}):\ V'\geqq 0\),
    \item \(q_{0}\in \sL^{\infty}(\R)\cap \sTV(\R)\),
    \item \(\gamma_{\eta}\in \sBV(\R_{\leq 0};\R_{\geq 0})\) satisfying either the assumptions on the kernel in \cite[Eq. (1.5) and (1.8)]{colombo2023nonlocal} or \cite[Assumption 2]{keimer2023singular}.
\end{itemize}
%\end{multicols}
\end{ass}
\begin{rem}[The meaning of the assumptions on the involved kernel]
Both classes of kernels must have \(L^{1}\) mass equal to \(1\) and be non-negative.
The assumptions also imply the following:
\begin{itemize}
\item The kernels satisfying  \cite[Eq. (1.5) and (1.8)]{colombo2023nonlocal} need to increase in a convex manner on \(\R_{\leq 0}\).
\item The conditions on the kernels in \cite[Assumption 2]{keimer2023singular} are more involved as these kernels are scaled in the exponent. This enables the use of non-monotonic and rather irregular kernels.
\end{itemize}
We refer the reader to the corresponding references for further details.
\end{rem}
To provide a better understanding of why the claimed convergence will hold, we point out that the considered problem class conserves the \(\sL^{1}\)-norm. This becomes evident when multiplying the conservation law in \cref{eq:local_dynamics} or \cref{eq:nonlocal_dynamics} by \(\sgn(q)\) or \(\sgn(q_{\eta}),\) respectively.

\section{The involved equations}
\subsection{The nonlocal equation}
As the type of nonlocal dynamics studied here has not been considered in the literature before, we will state an existence and uniqueness theorem, which also guarantees a (weak) form of a maximum principle on \(|q_{\eta}|\).
However, we first define what we mean by a weak solution.
\begin{defn}[Weak solution to the nonlocal conservation law]\label{defi:weak_solution_nonlocal}
    We call \(q_{\eta}\in \sC\big([0,T];\sL^{1}_{\loc}(\R)\big)\cap \sL^{\infty}((0,T);\sL^{\infty}(\R))\) a weak solution to the nonlocal conservation law in \cref{eq:nonlocal_dynamics} iff \(\forall \phi\in \sC^{1}_{\text{c}}((-42,T)\times\R)\) it holds that
\begin{align*}
&\iint_{\OT}q_{\eta}(t,x)\big(\phi_t(t,x)+\phi_x(t,x) V(W[|q_{\eta}|,\gamma_{\eta}](t,x))\big)\dd x\dd t \\
&\qquad +\int_{\R}\phi(0,x)q_{0}(x)\dd x=0.    
\end{align*}

\end{defn}
\begin{thm}[Existence and uniqueness of the nonlocal conservation law and a maximum principle]\label{theo:existence_uniqueness_max}
Let \cref{ass:input_datum} hold.
    Then, there exists a unique weak solution to \cref{eq:nonlocal_dynamics}, \[q_{\eta}\in \sC\big([0,T];\sL^{1}_{\loc}(\R)\big)\cap \sL^{\infty}((0,T);\sL^{\infty}\cap \sTV(\R))\] on any time horizon \(T\in\R_{>0}\). 
    Additionally, for the assumptions on the kernel, we have the following:
    \begin{itemize}\item  In \cite[Eq. (1.5) and (1.8)]{colombo2023nonlocal}, the solution satisfies a maximum principle
    \[
    \|q_{\eta}(t,\cdot)\|_{\sL^{\infty}(\R)}\leq \|q_{0}\|_{\sL^{\infty}(\R)}, \forall t\in[0,T].
    \]
\item In \cite[Assumption 2]{keimer2023singular}, a weak maximum principle holds: for any \((T,\kappa)\in\R_{>0}^{2}\) there exists \(\eta_{T,\kappa}\in \R_{>0}\) such that
\[
\forall (t,\eta)\in [0,T]\times(0,\eta_{T,\kappa}):\ \|q_{\eta}(t,\cdot)\|_{\sL^{\infty}(\R)}\leq (1+\kappa)\|q_{0}\|_{\sL^{\infty}(\R)},\ \forall t\in[0,T].
\]
\end{itemize}
Finally, the following approximation result holds. Assuming \(q_{0,\eps}\in C^{\infty}(\R)\cap \sTV(\R)\cap\sL^{\infty}(\R)\) and denoting with \(q_{\eta,\eps}\) the corresponding solution to the nonlocal equation, there exists \(C=C(T,\eta)\in\R_{>0}\) such that
\[
\|q_{\eta}-q_{\eta,\eps}\|_{\sC([0,T];\sL^{1}(\R))}\leq C\|q_{0}-q_{0,\eps}\|_{\sL^{1}(\R)}
\]
and \(q_{\eta,\eps}\in \sW^{1,\infty}(\OT)\), and as such a strong solution of the nonlocal conservation law in \cref{eq:nonlocal_dynamics}.
\end{thm}
\begin{proof}
We only sketch the proof.
The existence and uniqueness of weak solutions follow with classical fixed-point arguments laid out, for instance, in \cite{pflug,coclite2022existence}. The integral operator in the velocity function \(V\) renders the corresponding characteristics and the related system of initial value problems in the ODE Lipschitz-continuous so that on a small time horizon, it can be shown that a unique weak solution exists using classical fixed-point arguments.

The long-time-horizon existence is then a consequence of the maximum principle, which is addressed later. To this end, one can derive a related equation in the absolute value of the original solution. Indeed, multiplying \cref{eq:nonlocal_dynamics} by \(\sgn(q_{\eta}(t,x)),\ (t,x)\in\OT\), one obtains, formally,
    \begin{align*}
        \partial_{t}|q_{\eta}(t,x)|+\partial_{x}\big(V(W[|q_{\eta}|,\gamma_{\eta}](t,x))|q_{\eta}(t,x)|\big)&=0,&& (t,x)\in\OT\\
        |q(0,x)|&=|q_{0}(x)|,&& x\in\R.
    \end{align*}
    Thus, the quantity \(|q_{\eta}|\) satisfies the well-known sign-restricted nonlocal conservation law studied in, e.g., \cite{pflug,scialanga,colombo2023nonlocal}.
    In particular, for the kernel in \cite[Eq. (1.5) and (1.8)]{colombo2023nonlocal}, the claimed maximum principle for \(|q_{\eta}|\) (see, e.g., \cite[Proposition 2.1]{colombo2023nonlocal}, \cite[Theorem 2.1]{scialanga}, \cite[Corollary 2.1]{coclite2022existence}, \cite[Corollary 4.3, Remark 4.4]{pflug}) holds.

For the weak form of maximum principle that applies with a kernel as in \cite[Assumption 2]{keimer2023singular}, we refer the reader to \cite[Theorem 6]{keimer2023singular}. The maximum principle on \(|q_{\eta}|\) carries over to the maximum/minimum principle for \(q_{\eta}\).
For the claimed stability and regularity of \(q_{\eta,\eps}\) we refer the reader to \cite[Theorem 3]{keimer2023singular} and to \cite[Theorem 3.1]{pflug4}, concluding the sketch of the proof.
\end{proof}
\subsection{The local equation}
Concerning the theory of local conservation laws, we give in the following the definition of weak and entropy solutions, and the basic result on existence and uniqueness.

\begin{defn}[Weak solution for the local equation]\label{defi:weak_solution_local}
We call a function \(q^{*}\in \sL^{\infty}((0,T);\sL^{\infty}(\R))\) a weak solution to \cref{eq:local_dynamics} for the initial datum \(q_{0}\in \sL^{\infty}\cap \sTV(\R)\) iff \(\forall\phi\in \sC^{1}_{\text{c}}((-42,T)\times\R)\) the following integral equation holds:
\[
\iint_{\OT}\!\!\!\!\!\!q^{*}(t,x) \big(\partial_{t}\phi(t,x)+\partial_{x}\phi(t,x) V(|q^{*}(t,x)|)\big)\dd x\dd t+\!\!\int_{\R}\!\!\phi(0,x)q_{0}(x)\dd x=0.
\]
\end{defn}
As is well known, weak solutions are not unique, which is why one postulates an entropy condition (which also contains the weak formulation).
\begin{defn}[Entropy solution]\label{defi:entropy}
We call the function \(q\in \sC([0,T];\sL^{1}_{\loc}(\R))\cap \sL^{\infty}((0,T);\sL^{\infty}(\R))\) a weak entropy solution to \cref{eq:local_dynamics} iff
it satisfies the following entropy inequality
for all convex \(\alpha\in\sC^{2}(\R)\) and \(\beta\in \sC^{1}(\R)\) satisfying \(\beta'\equiv\alpha'\cdot f',\) where
\[
f(x)=xV(x),\ \forall x\in\R,
\] and for all \(\phi\in \sC_{\text{c}}^{1}\big((-42,T)\times\R;\R_{\geq0}\big)\):
\begin{equation}
\begin{aligned}
\mathcal{E}[\phi,\alpha,q]&\coloneqq\iint_{\OT}\!\!\!\!\alpha(q(t,x))\phi_{t}(t,x)+\beta(q(t,x))\phi_{x}(t,x)\dd x\dd t\\
&\quad+\int_{\R}\!\!\alpha(q_{0}(x))\phi(0,x)\dd x\geq 0.
\end{aligned}
\label{eq:Entropy}
\end{equation}
\end{defn}
\begin{thm}[Existence and uniqueness of local solutions]\label{theo:existence_uniqueness_local}
    Let the (local) conservation law in \cref{eq:local_dynamics} be given with the assumptions in \cref{ass:input_datum}. Then, there exists on every time horizon \(T\in\R_{>0}\) a unique weak entropy solution in the sense of \cref{defi:entropy}.
\end{thm}
\begin{proof}
This can be found in \cite[Theorem 6.3]{bressan}, \cite[Theorem 19.1]{eymard}. For Kru{\v{z}}kov entropies and entropy flux pairs, we refer to
 \cite[Theorem 2, Theorem 5, Section 5 Item 4]{kruzkov}, or \cite{godlewski}.
\end{proof}

\section{The singular limit problem}

As we are interested in the convergence of \(q_{\eta}\) for \(\eta\rightarrow 0\), a compactness result on the nonlocal term is required, which is stated in the following lemma.
\begin{lem}[Compactness of {\(W_{\eta}[|q_{\eta}|,\gamma_{\eta}]\)} in {\(\sC([0,T];\sL^{1}_{\loc}(\R))\)}]\label{lem:compactness}
 Given the assumptions of \cref{theo:existence_uniqueness_max}, 
 and assuming that the kernel satisfies 
 \cite[Eq. (1.5) and (1.8)]{colombo2023nonlocal} or \cite[Assumption 2]{keimer2023singular}, there exists for every \(T\in\R_{>0}\) an \(\eps\in\R_{>0}\) such that
  \[
\big\{W[|q_{\eta}|,\gamma_{\eta}]\in \sC\big([0,T];\sL^{1}_{\loc}(\R)\big):\ \eta\in(0,\eps)\big\}\overset{\text{c}}{\hookrightarrow}\sC\big([0,T];\sL^{1}_{\loc}(\R)\big).
 \]
  \end{lem}
\begin{proof}
This is a consequence of the \(\sTV\) bounds uniform in \(\eta\in\R_{>0}\) for both classes of kernels (see \cite[Theorem 1]{colombo2023nonlocal} and \cite[Theorem 7]{keimer2023singular}) together with the uniform bounds of the solutions in \(\eta\), thanks to the maximum principle \cref{theo:existence_uniqueness_max}. Together with an equicontinuity in time as required by \cite[Lemma 1]{simon}, and which is also a consequence of the uniform \(\sTV\) estimate, the named compactness follows (as also explicitly stated in \cite[Theorem 8]{keimer2023singular}).
\end{proof}

Equipped with the results from the previous section, we show here that the solution to the nonlocal equation converges to the local entropy solution when the nonlocal operator converges to a point evaluation. We start with a short lemma showing the convergence to a weak local solution not only of the nonlocal operator but also weak-star convergence of the solution itself.
\begin{lem}[Convergence to a weak local solution]\label{lem:convergence_weak_solution}
Let \cref{ass:input_datum} be given. Then for any \(\eta_{k}\in\R_{>0},\ k\in\N_{\geq 1},\lim_{k\rightarrow\infty}\eta_{k}=0\), there exists a subsequence (still denoted by \(\eta_{k}\)) such that the nonlocal operator converges in the following sense:
\[
\lim_{k\rightarrow\infty}\|W[|q_{\eta}|,\gamma_{\eta}]-|q^{*}|\|_{\sC([0,T];\sL^{1}_{\loc}(\R))}=0,
\]
as well as
\[
q_{\eta}\weakstar q^{*}\ \text{ in }\ \sL^{\infty}(\OT),
\]
where \(q^{*}\in \sC\big([0,T];\sL^{1}_{\loc}(\R)\big)\) is a weak solution to the local equation.
\end{lem}
\begin{proof}
Thanks to the uniform boundedness of \(q_{\eta}\), as stated in \cref{theo:existence_uniqueness_max}, we can invoke the Banach--Alaoglu--Bourbaki theorem \cite[Theorem 3.16]{brezis} and we obtain that there exists a subsequence (still denoted by \(\eta\)) and a \(q^{*}\in \sL^{\infty}((0,T);\sL^{\infty}(\R))\) such that
\[
q_{\eta}\weakstar q^{*}\ \text{ in }\  \sL^{\infty}((0,T);\sL^{\infty}(\R)).
\]
Using this subsequence in the nonlocal operator, we can apply the compactness result in \cref{lem:compactness} to deduce that there exists another subsequence, denoted by \(\big(\eta_{k}\big)_{k\in\N}\subset\R_{>0},\ \lim_{k\rightarrow\infty} \eta_{k}=0\) such that
\[
\lim_{k\rightarrow\infty}\|W[|q_{\eta_{k}}|,\gamma_{\eta_{k}}]-|q^{*}|\|_{\sC([0,T];\sL^{1}_{\loc}(\R))}=0.
\]
Next, we need to show that this limit point \(q^{*}\) is a weak local solution. For \(\eta\in\R_{>0}\), we have that \(q_{\eta_{k}}\) and the corresponding nonlocal operator satisfy 
\[\iint_{\OT}q_{\eta_{k}} \big(\phi_t+\phi_x V(W[|q_{\eta_{k}}|,\gamma_{\eta}])\big)\dd x\dd t+\int_{\R}\phi(0,\cdot)q_{0}\dd x=0.\]
However, thanks to the strong \(\sC(\sL^{1})\) convergence of the nonlocal operator, \(V\in \sW^{1,\infty}_{\loc}(\R)\) and \(q_{\eta}\) weak-star converges in \(\sL^{\infty}(\sL^{\infty})\), and thanks to the fact that the product of a strongly converging sequence and a weak-star converging sequence converges weak-star, we obtain in the limit \(k\rightarrow\infty\) that
\[
\iint_{\OT}q^{*} \big(\phi_t+\phi_x V(|q^{*}|)\big)\dd x\dd t+\int_{\R}\phi(0,\cdot)q_{0}\dd x=0,
\]
which is the weak formulation for the local conservation law, as stated in \cref{defi:weak_solution_local}. Thus, \(q^{*}\) is a weak solution to the local conservation law, concluding the proof.
\end{proof}\begin{rem}[Convergence and weak solutions]
    As is well known, the convergence to a weak solution of the local equation follows with classical compactness results. However, these weak solutions are in the local case not unique, requiring some well-known and established entropy conditions to single out the ``reasonable'' unique solution. This is what we will discuss in the following.
\end{rem}
\subsection{Entropy admissibility}
The following calculations will lead to the crucial result, that if \(q_{\eta}\) converges to a limit then it satisfies the entropy admissibility condition.
However, we will present these calculations only for the one-sided exponential kernel, strictly convex ``flux'' and provide some remarks about how this can be generalized to the arbitrary kernels. In \cref{fig:convergence}, the mentioned convergence is numerically exemplified for two initial data and velocity functions.

\begin{figure}
    \centering
  \includegraphics[width=.75\textwidth,clip,trim=70 380 65 60]{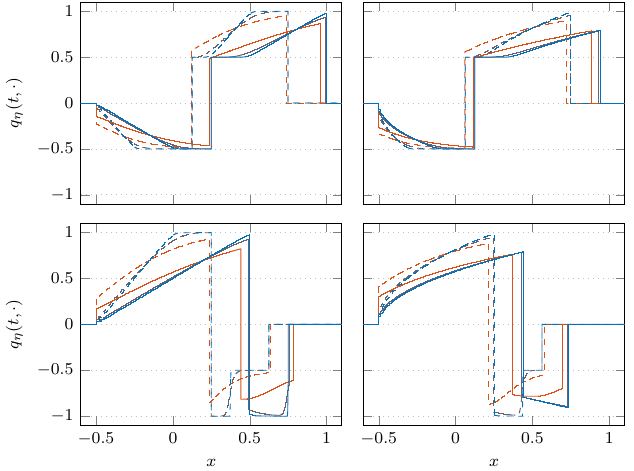}

\caption{Solution of the nonlocal balance law stated in \cref{eq:nonlocal_dynamics}  for $\eta \in  \{10^{-1},10^{-2},10^{-3}\}$ (\textbf{\textcolor{col2}{red}} to \textbf{\textcolor{col1}{blue}}) for initial datum $q_0 \equiv -\tfrac{1}{2}\chi_{(-0.5,0)}+\chi_{(0,0.5)}$ (\textbf{top}) and $q_0 \equiv \tfrac{1}{2}\chi_{(-0.5,0)}-\chi_{(0,0.5)}$ (\textbf{bottom}), and velocity $V \equiv \text{Id}$ (\textbf{left}) and $V \equiv \text{Id}^2$ (\textbf{right}). This illustrates the proven convergence for $\eta \rightarrow 0$ to entropy solutions to the local counterpart of the nonlocal conservation law, as stated in \cref{eq:local_dynamics}.}

\label{fig:convergence}

\end{figure}

\begin{thm}[Nonlocal solution entropy admissible in the limit/convergence]\label{theo:convergence_entropy}
Choose as nonlocal kernel \(\gamma_{\eta}\equiv\exp(\cdot/\eta)/\eta\) and consider the corresponding nonlocal solution \(q_{\eta}\in \sC\big([0,T];\sL^{1}_{\loc}(\R)\big)\) of \cref{eq:nonlocal_dynamics}.
Assume in addition that for \(\eps\in\R_{>0}\), the flux \(x\mapsto xV(x),\ x\in [-\|q_{0}\|_{\sL^{\infty}(\R)}-\eps,\|q_{0}\|_{\sL^{\infty}(\R)}+\eps]\), is strictly concave.
Then, this solution \(q_{\eta}\) is entropy admissible in the singular limit, i.e., it satisfies \(\forall \phi,\alpha\), as in \cref{defi:entropy},
\[
\lim_{\eta\rightarrow 0}\mathcal{E}[\phi,\alpha,q_{\eta}]\geq 0.
\]
In particular, the nonlocal \(q_{\eta}\) converges weak-star in \(\sL^{\infty}((0,T);\sL^{\infty}(\R))\) to the local entropy solution \(q^{*}\), and the nonlocal operator \(W[|q_{\eta}|,\gamma_{\eta}]\) converges strongly in \(\sC\big([0,T];\sL^{1}_{\loc}(\R)\big)\) to \(|q^{*}|\).
\end{thm}
\begin{proof}
We first remark that in the case of the considered exponential kernel, we obtain the following identity \(\forall (t,x)\in\OT\):
\begin{equation}
\partial_{x}W_\eta (t,x)=\tfrac{1}{\eta}\big(|q_{\eta}(t,x)|-W_\eta (t,x)\big).\label{eq:nonlocal_identity}
\end{equation} Here we use the abbreviation $W_\eta \coloneqq W_\eta[|q_\eta|,\tfrac{1}{\eta}\exp(\tfrac{\cdot}{\eta})]$,
which will be useful later.

Assuming sufficient regularity on the nonlocal equation, which is possible thanks to the stability in \(\sL^{1}\) (see again \cref{theo:existence_uniqueness_max}), we can perform the following manipulations when plugging \(q_{\eta}\) into the entropy admissibility condition:
\begin{align}
&\mathcal{E}[\phi,\alpha,q_{\eta}] \notag\\
&=-\iint_{\OT}\Big(\alpha'(q_{\eta})\partial_{t}q_{\eta}+ \beta'(q_{\eta})\partial_{x}q_{\eta}\Big)\phi\dd x\dd t\notag\\
&=-\iint_{\OT}\!\!\!\!\!\alpha'(q_{\eta})\Big(\partial_{t}q_{\eta}+ f'(q_{\eta})\partial_{x}q_{\eta}\Big)\phi\dd x\dd t\notag\\
&=\iint_{\OT}\!\!\!\!\!\alpha'(q_{\eta})\Big(\partial_{x}\big(V(W_{\eta})q_{\eta}\big)- V(|q_{\eta}|)\partial_{x}q_{\eta}-V'(|q_{\eta}|)\partial_{x}|q_{\eta}|q_{\eta}\Big)\phi\dd x\dd t \notag\\
&=\iint_{\OT}\alpha''(q_{\eta})\partial_{x}q_{\eta}\big(V(|q_{\eta}|)-V(W_{\eta})\big)q_{\eta}\phi\dd x\dd t \label{eq:42}\\
&\quad +\iint_{\OT}\alpha'(q_{\eta})\big(V(|q_{\eta}|)-V(W_{\eta})\big)q_{\eta}\phi_{x}\dd x\dd t.\notag
\end{align}
To show a lower bound for $\mathcal E,$ we first focus on the first integral in the latter expression. Take \(\alpha''=1\) (this is not a restriction for the entropy admissibility as we can take advantage of the results in \cite{otto,panov1994uniqueness}, which guarantee the uniqueness of solutions when only using one entropy flux pair) to obtain
\begin{align*}
\eqref{eq:42}&=\iint_{\OT}\partial_{x}q_{\eta}\big(V(|q_{\eta}|)-V(W_{\eta})\big)q_{\eta}\phi\dd x\dd t\\
    %&=\tfrac{1}{2}\iint_{\OT}\tfrac{\dd}{\dd x}q_{\eta}^{2}V(|q_{\eta}|)\phi\dd x\dd t -\iint_{\OT} V(W_{\eta})\partial_{x}q_{\eta}q_{\eta}\phi\dd x\dd t\\
    %&=\tfrac{1}{2}\iint_{\OT}\tfrac{\dd}{\dd x}|q_{\eta}|^{2}V(|q_{\eta}|)\phi\dd x\dd t -\iint_{\OT} V(W_{\eta})\partial_{x}q_{\eta}q_{\eta}\phi\dd x\dd t\\
    &=\iint_{\OT}V(|q_{\eta}|)|q_{\eta}|\partial_{x}|q_{\eta}|\phi\dd x\dd t -\iint_{\OT} V(W_{\eta})q_{\eta}\partial_{x}q_{\eta}\phi\dd x\dd t.
    \intertext{With $G(y) \coloneqq \int_0^y sV(s) \dd s$ and by integration by parts in the first integral, we continue:}
    \eqref{eq:42}&=-\iint_{\OT} G(|q_{\eta}|)\phi_{x}\dd x\dd t-\iint_{\OT} V(W_{\eta})q_{\eta}\partial_{x}q_{\eta}\phi\dd x\dd t\\
    &=-\iint_{\OT} \big(G(|q_{\eta}|)-G(W_\eta)\big)\phi_{x}\dd x\dd t -\iint_{\OT} G(W_{\eta})\phi_{x}\dd x\dd t\\
    &\quad -\iint_{\OT} V(W_{\eta})q_{\eta}\partial_{x}q_{\eta}\phi\dd x\dd t.\\
    \intertext{Another integration by parts in the second term yields}
    \eqref{eq:42}&=-\iint_{\OT} \big(G(|q_{\eta}|)-G(W_\eta)\big)\phi_{x}\dd x\dd t \\
    &\quad +\iint_{\OT} V(W_{\eta})W_{\eta}\partial_{x}W_{\eta}\phi\dd x\dd t -\iint_{\OT} V(W_{\eta})q_{\eta}\partial_{x}q_{\eta}\phi\dd x\dd t,
    \intertext{and integration by parts in the second and third expressions leads to}
    \eqref{eq:42}&= -\iint_{\OT} \big(G(|q_{\eta}|)-G(W_\eta)\big)\phi_{x}\dd x\dd t\\
    &\quad  -\tfrac{1}{2}\iint_{\OT} V(W_{\eta})W_{\eta}^{2}\phi_x\dd x \dd t -\tfrac{1}{2}\iint_{\OT}V'(W_{\eta})W_{\eta}^{2} \partial_{x} W_{\eta}\phi\dd x \dd t\\
    &\quad +\tfrac{1}{2}\iint_{\OT} V'(W_{\eta})\partial_{x}W_{\eta}q_{\eta}^{2}\phi\dd x\dd t +\tfrac{1}{2}\iint_{\OT} V(W_{\eta})q_{\eta}^{2}\phi_x\dd x\dd t\\
    &= -\iint_{\OT}\!\!\!\!\! \big(G(|q_{\eta}|)-G(W_\eta)\big)\phi_{x}\dd x\dd t  +\tfrac{1}{2}\!\!\iint_{\OT}\!\!\!\!\!\big(q_{\eta}^{2}-W_{\eta}^{2}\big) V(W_{\eta})\phi_x\dd x \dd t\\
    &\quad -\tfrac{1}{2}\iint_{\OT}W_{\eta}^{2} V'(W_{\eta})\partial_{x} W_{\eta}\phi\dd x \dd t+ \tfrac{1}{2}\iint_{\OT} V'(W_{\eta})\partial_{x}W_{\eta}q_{\eta}^{2}\phi\dd x\dd t.\\
    \intertext{Using the identity for the nonlocal operator in \cref{eq:nonlocal_identity}\(|q_{\eta}|=W_{\eta}+\eta \partial_{x}W_{\eta}\), which implies \(q_{\eta}^{2}=W_{\eta}^{2}+2\eta W_{\eta}\partial_{x}W_{\eta}+\eta^{2}\big(\partial_{x}W_{\eta}\big)^{2}\), in the last term yields}
    \eqref{eq:42}&= -\iint_{\OT}\!\!\!\!\! \Big(G(|q_{\eta}|)-G(W_\eta)\Big)\phi_{x}\dd x\dd t  +\tfrac{1}{2}\!\!\iint_{\OT}\!\!\!\!\!\big(q_{\eta}^{2}-W_{\eta}^{2}\big) V(W_{\eta})\phi_x\dd x \dd t\\
    &\quad -\tfrac{1}{2}\iint_{\OT} V'(W_{\eta})\partial_{x} W_{\eta}W_{\eta}^{2}\phi\dd x +\tfrac{1}{2}\iint_{\OT} V'(W_{\eta})\partial_{x}W_{\eta}W_{\eta}^{2}\phi\dd x\dd t\\
    &\quad +\eta\!\iint_{\OT}\!\!\!\! V'(W_{\eta})\big(\partial_{x}W_{\eta}\big)^{2}W_{\eta}\phi\dd x\dd t +\tfrac{\eta^{2}}{2}\!\!\iint_{\OT}\!\!\!\! V'(W_{\eta})\big(\partial_{x}W_{\eta}\big)^{3}\phi\dd x\dd t\\
    &= -\iint_{\OT}\!\!\!\!\! \big(G(|q_{\eta}|)-G(W_\eta)\big)\phi_{x}\dd x\dd t  +\tfrac{1}{2}\!\!\iint_{\OT}\!\!\!\!\!\big(q_{\eta}^{2}-W_{\eta}^{2}\big) V(W_{\eta})\phi_x\dd x \dd t\\
    &\quad +\eta\iint_{\OT} V'(W_{\eta})\big(\partial_{x}W_{\eta}\big)^{2}\big(W+\tfrac{\eta}{2}W_{x}\big)\phi\dd x\dd t.
    \intertext{Again, thanks to the identity in \cref{eq:nonlocal_identity}, as well as \(|q_{\eta}|=W_{\eta}+\eta\partial_{x}W_{\eta}\) and \(V'\geqq0\), we can estimate the last term by zero from below and obtain}
    \eqref{eq:42}&\geq -\iint_{\OT}\!\!\!\!\! \big(G(|q_{\eta}|)-G(W_\eta)\big)\phi_{x}\dd x\dd t  +\tfrac{1}{2}\!\!\iint_{\OT}\!\!\!\!\!\big(q_{\eta}^{2}-W_{\eta}^{2}\big) V(W_{\eta})\phi_x\dd x \dd t.
\end{align*}
Thus, in total, we obtain
\begin{align*}
\mathcal{E}[\phi,\alpha,q_{\eta}] &\geq -\iint_{\OT} \big(G(|q_{\eta}|)-G(W_\eta)\big)\phi_{x}\dd x\dd t  \\
&\quad +\tfrac{1}{2}\iint_{\OT}\big(q_{\eta}-W_{\eta}\big)\big(q_{\eta}+W_{\eta}\big)  V(W_{\eta})\phi_x\dd x \dd t \\
&\quad + \iint_{\OT}q_\eta^2\big(V(|q_{\eta}|)-V(W_{\eta})\big)\phi_{x}\dd x\dd t\\
    &\overset{\eta\rightarrow 0}{=}0,
\end{align*}
which holds when recalling that, by definition, we have the following:
\begin{itemize}
    \item  \(V,G\in \sW^{1,\infty}_{\loc}(\R)\) according to \cref{ass:input_datum} and $G(\cdot)\coloneqq \int_0^\cdot xV(x) \dd x$,
    \item  \(\||q_{\eta}|-W_{\eta}\|_{\sC([0,T];\sL^{1}_{\loc}(\R))}\overset{\eta\rightarrow 0}{\longrightarrow} 0\) according to \cref{lem:convergence_weak_solution},
    \item  \(q_{\eta}\) and \(W_\eta\) are uniformly bounded in \(\eta\) thanks to \cref{theo:existence_uniqueness_max}.
\end{itemize}
Up to now, this convergence only holds for a suitable subsequence. But as the local entropy solution is unique (see \cref{theo:existence_uniqueness_local}) and we can show for all subsequences that there exists a subsequence converging to the local entropy solution, i.e., to the same limit, we arrive at the convergence to the entropy solution for any sequence.
This concludes the proof.
\end{proof}

\begin{rem}[Generalization] \label{rem:generalization}
    Similarly to the proof of the latter theorem, one can show, as in \cite{keimer2023singular} and \cite{colombo2023nonlocal}, that the convergence to the entropy solution in \cref{theo:convergence_entropy} holds for all kernels defined in \cref{ass:input_datum} and that the strict concavity assumption can be removed as well, so that in the end, we obtain the general convergence result for the problem class in \cref{eq:nonlocal_dynamics}.
\end{rem}
%\end{document}

\subsection*{Acknowledgment}
L. Pflug acknowledges funding by the Deutsche Forschungsgemeinschaft (German Research Foundation)---Project-ID 416229255---SFB 1411.

\bibliographystyle{plain}
\bibliography{biblio}

% ------------------------------------------------------------------------
\end{document}